\documentclass[leqno,12pt]{amsart}
\usepackage{amsfonts}
\usepackage{amsmath}
\usepackage{amssymb}
\usepackage{graphicx}
\newtheorem{theorem}{Theorem}[section]
\theoremstyle{plain}

\newtheorem{definition}{Definition}[section]
\newtheorem{example}{Example}[section]

\newtheorem{lemma}{Lemma}[section]

\newtheorem{proposition}{Proposition}[section]
\newtheorem{remark}{Remark}[section]

\numberwithin{equation}{section} \textheight  22 true cm \textwidth  15 true cm  

\begin{document}
\title[conformal generic submersions ...]{\small{Conformal Generic Submersions With Total Space 
an almost Hermitian manifold}}
\author[Akyol]{M\lowercase{ehmet} A\lowercase{kif} AKYOL }
\address{Bing\"{o}l University, Faculty of Science and Arts, Department of Mathematics, 12000, Bing\"{o}l, Turkey}
\email{mehmetakifakyol@bingol.edu.tr}

\begin{abstract}
Akyol, M. A and \c{S}ahin, B. [Conformal semi-invariant submersions, Commun. Contemp. Math. 19, 1650011 (2017).] 
introduced the notion  of conformal semi-invariant submersions from almost Hermitian manifolds. 
The present paper deal with the study of conformal generic submersions from almost
Hermitian manifolds which extends semi-invariant submersions, generic Riemannian 
submersions and conformal semi-invariant submersions a natural way. 
We mention some examples of such maps and
obtain characterizations and investigate some properties, including
the integrability of distributions, the geometry of foliations and totally geodesic foliations. Moreover, we obtain some conditions  
for such submersions to be totally geodesic and harmonic, respectively.
\end{abstract}

\maketitle

\section{\textbf{Introduction}}

Let $\tilde{M}$ be an almost Hermitian manifold with almost complex structure $%
J$ and $M$ a Riemannian manifold isometrically immersed in $\tilde{M}.$ We
note that submanifolds of a K\"{a}hler manifold are determined by the
behavior of tangent bundle of the submani-fold under the action of the almost
complex structure of the ambient manifold. A submanifold $M$ is called
holomorphic(complex) if $J(T_{q}M)\subset T_{q}M,$ for every $q\in M,$ where
$T_{q}M$ denotes the tangent space to $M$ at the point $q.$ $M$ is called
totally real if $J(T_{q}M)\subset T^{\perp}_{q}M,$ for every $q\in M, $
where $T^{\perp}_{q}M$ denotes the normal space to $M$ at the point $q.$ As
a generalization of holomorphic and totally real submanifolds,
$CR-$submanifolds were introduced by Bejancu \cite{Bejancu}. A $CR-$submanifold $%
M$ of an almost Hermitian manifold $\tilde{M}$ with an almost complex
structure $J$ requires two orthogonal complementary distributions $\mathcal{D}$ and $%
\mathcal{D}^{\perp}$ defined on $M$ such that $\mathcal{D}$ is invariant under $J$ and
$\mathcal{D}^{\perp} $ is totally real \cite{Bejancu}. There is yet another
generalization of $CR-$submanifolds known as generic submanifolds
\cite{Chen}. These submanifolds are defined by relaxing the
condition on the complementary distribution of holomorphic
distribution. Let $M$ be a real submanifold of an almost Hermitian
manifold $\tilde{M},$ and let $\mathcal{D}_q=T_{q}M\cap JT_{q}M$ be the maximal
holomorphic subspace of $T_{q}M.$ If $\mathcal{D}:q\longrightarrow \mathcal{D}_q$
defines a smooth holomorphic distribution on $M,$ then $M$ is called a
generic submanifold of $\tilde{M}.$ The complementary distribution
$\mathcal{D}^\perp$ of $\mathcal{D}$ is called purely real distribution on $M.$ A
generic submanifold is a $CR-$submanifold if the purely real
distribution on $M$ is totally real. A purely real distribution
$\mathcal{D}^\perp$ on a generic submanifold $M$ is called proper if it is not
totally real. A generic submanifold is called proper if purely real
distribution is proper. Generic submanifolds have been studied
widely by many authors and the theory of such submanifolds is still
an active research area, see \cite{CKT}, \cite{DSC}, \cite{DS},
\cite{KCT}, \cite{Kon}, \cite{VilcuCCM} for  recent papers on this
topic.

The notion of Riemannian submersions between Riemannian manifolds were studied by
O'Neill \cite{O} and Gray \cite{Gray}. Later on, such submersions
have been studied widely in differential geometry. Riemannian
submersions between Riemannian manifolds equipped with an additional
structure of almost complex type was firstly studied by Watson
\cite{W}. Watson defined an almost Hermitian submersion between
almost Hermitian manifolds and he showed that the base manifold and
each fiber have the same kind of structure as the total space, in
most cases. We note that almost Hermitian submersions have been
extended to the almost contact manifolds \cite{Chi3}, locally
conformal K\"{a}hler manifolds \cite{MR}, quaternionic K\"{a}hler
manifolds \cite{IMV}, paraquaternionic manifolds \cite{C},
\cite{V} and statistical manifolds \cite{Vilcu-Vilcu}. 

Recently, \c{S}ahin \cite{S2} introduced the notion of
semi-invariant Riemannian submersions as a generalization
of anti-invariant Riemannian submersions \cite{S1} from almost Hermitian manifolds 
onto Riemannian manifolds. Later such submersions and their extensions
are studied \cite{Akyol1}, \cite{AG},  \cite{OST}, \cite{Park}, \cite{Park1}, \cite{S4} and \cite{S5}.
As a generalization of semi-invariant submersions, Ali and Fatima \cite{Shahid} 
introduced the notion of generic Riemannian submersions. (see also \cite{Akyol}).

On the other hand, A related topic of growing interest deals with the study of the so-called horizontally conformal submersions: 
these maps, which provide a natural generalization of Riemannian submersion, introduced independently Fuglede \cite{F} and Ishihara \cite{I}.
 As a generalization of holomorphic submersions, the notion of conformal holomorphic submersions were defined by Gudmundsson and Wood \cite{GW}.
 In 2017, Akyol and \c{S}ahin \cite{As} defined a conformal semi-invariant submersion from an almost Hermitian manifolds onto a riemannian manifold. 
 In this paper, we introduce conformal generic
submersions as a generalization of semi-invariant submersions, generic Riemannian submersions
and conformal semi-invariant submersions, investigate the geometry
of the total space and the base space for the existence of such
submersions.

The present article is organized as follows. 
In Section 2, we give some background about conformal submersions and the second fundamental maps. 
In Section 3, we define and study conformal generic submersions from almost Hermitian manifolds onto
Riemannian manifolds, give examples and investigate the geometry of leaves
of the horizontal distribution and the vertical distribution. In this
section we also show that there are certain product structures on the total
space of a conformal generic submersion. In the last section of this paper, we find necessary and
sufficient conditions for a conformal generic submersion to be totally
geodesic and harmonic, respectively.

\section{\textbf{Preliminaries}}
The manifolds, maps, vector fields etc. considered in this paper are assumed to be smooth, i.e. differentiable of class $C^{\infty}$.
\subsection{Conformal submersions}
Let $\psi:(M_1,g_{1})\longrightarrow (M_2,g_{2})$ be a smooth map between Riemannian manifolds, and let $p\in M_1$. Then $\psi$ is
called \textit{horizontally weakly conformal or semi conformal} at $p$ \cite{BW} if either (i) $d\psi_{p}=0$, or
(ii) $d\psi_{p}$ is surjective and there exists a number $\Lambda(p)\neq0$ such that
$$
g_2(d\psi_{p}\xi,d\psi_{p}\eta)=\Lambda(p)g_1(\xi,\eta)\mbox{ }(\xi,\eta\in\mathcal{H}_{p}).
$$
We call the point $p$ is of type (i) as a critical point if it satisfies the type (i), and we shall call the point $p$ a regular point if it satisfied the type (ii).
At a critical point, $d\psi_{p}$ has rank $0$; at a regular point, $d\psi_{p}$ has rank $n$
and $\psi$ is submersion. Further, the positive number $\Lambda(p)$ is called the \textit{square dilation} (of $\psi$ at $p$). The map $\psi$ is called \textit{horizontally weakly
conformal} or \textit{semi conformal} (on $M_1$) if it is horizontally weakly
conformal at every point of $M_1$ and it has no critical point, then we call it a (\textit{horizontally conformal submersion}).

A vector field $\xi_{1}\in\Gamma(TM_1)$ is called a \textit{basic vector field} if $\xi_{1}\in\Gamma(({\rm kerd}\psi)^\perp)$ and $\psi-$related with a vector field $\bar{\xi_{1}}\in\Gamma(TM_2)$ which means that $({{\rm d}\psi}_{p}\xi_{1p})=\bar{\xi_{1}}({\rm d}\psi(p))\in\Gamma(TM_2)$ for any $p\in\Gamma(TM_1).$

Define $\mathcal{T}$ and $\mathcal{A},$ which are O'Neill's tensors, as follows
\begin{equation}
\mathcal{A}_{E_1}E_2=\mathcal{V}\nabla^{1}_{\mathcal{H}E_1}\mathcal{H}E_2+\mathcal{H}%
\nabla^{1} _{\mathcal{H}E_1}\mathcal{V}E_2  \label{A}
\end{equation}%
\begin{equation}
\mathcal{T}_{E_1}E_2=\mathcal{H}\nabla^{1}_{\mathcal{V}E_1}\mathcal{V}E_2+\mathcal{V}%
\nabla^{1} _{\mathcal{V}E_1}\mathcal{H}E_2  \label{T}
\end{equation}
where $\mathcal{V}$ and $\mathcal{H}$ are the vertical and horizontal
projections (see \cite{FIP}). On the other hand, from (\ref{A}) and (\ref{T}), we have%
\begin{equation}
\nabla^{1}_{V}W=\mathcal{T}_{V}W+\hat{\nabla}_{V}W  \label{nvw}
\end{equation}%
\begin{equation}
\nabla^{1}_{V}\xi=\mathcal{H}\nabla^{1}_{V}\xi+\mathcal{T}_{V}\xi  \label{nvx}
\end{equation}%
\begin{equation}
\nabla^{1}_{\xi}V=\mathcal{A}_{\xi}V+\mathcal{V}\nabla^{1}_{\xi}V  \label{nxv}
\end{equation}%
\begin{equation}
\nabla^{1}_{\xi}\eta=\mathcal{H}\nabla^{1}_{\xi}\eta+\mathcal{A}_{\xi}\eta  \label{nxy}
\end{equation}
for $\xi,\eta\in\Gamma(({\rm kerd}\psi)^{\perp})$ and $V,W\in\Gamma({\rm kerd}\psi),$
 where $\hat{\nabla}_{V}W=\mathcal{V}\nabla^{1}_{V}W$. If
$\xi$ is basic, then $\mathcal{H}\nabla^{1}_{V}\xi=\mathcal{A}_{\xi}V$.

It is easily seen that for $q\in M_1,$ $\xi\in\mathcal{H}_{q}$ and $V \in
\mathcal{V}_{q}$ the linear operators $\mathcal{T}_{V},\mathcal{A}_{\xi}:T_{q}M_1\rightarrow T_{q}M_1$ are skew-symmetric, that is%
\begin{equation*}
-g_1(\mathcal{T}_{V}E_1,E_2)=g_1(E_1,\mathcal{T}_{V}E_2)\text{ and }-g_1(\mathcal{A}_{\xi}E_1,E_2)=g_1(E_1,\mathcal{A}_{\xi}E_2)
\end{equation*}
for all $E_1,E_2\in T_{q}M_1$. We also see that the restriction of $\mathcal{T}$ to the
vertical distribution $\mathcal{T}\mid_{{\rm kerd}\psi\times {\rm kerd}\psi}$ is exactly the
second fundamental form of the fibres of $\psi$. Since $\mathcal{T}_{V}$ is
skew-symmetric we get: $\psi$ has totally geodesic fibres if and only if $\mathcal{T}\equiv0$.\newline

Let $(M_1,g_{1})$ and $(M_1,g_{2})$ be Riemannian manifolds and suppose that $%
\psi:M_1\rightarrow M_2$ is a smooth map between them. Then the differential of
${\rm d}\psi$ of $\psi$ can be viewed a section of the bundle $%
Hom(TM_1,\psi^{-1}TM_2)\rightarrow M_1,$ where $\psi^{-1}TM_2$ is the pullback
bundle which has fibres $(\psi^{-1}TM_2)_{p}=T_{\psi(p)}M_2,$ $p\in M_1$. $%
Hom(TM_1,\psi^{-1}TM_2)$ has a connection $\nabla$ induced from the Levi-Civita
connection $\nabla^{M_1}$ and the pullback connection. Then the second
fundamental form of $\psi$ is given by%
\begin{equation}
(\nabla {\rm d}\psi)(\xi,\eta)=\nabla_{\xi}^{\psi}{\rm d}\psi(\eta)-{\rm d}\psi(\nabla^{1}_{\xi}\eta)  \label{nfixy}
\end{equation}
for $\xi,\eta\in\Gamma(TM_1),$ where $\nabla^{\psi}$ is the pullback connection. It
is known that the second fundamental form is symmetric.

\begin{lemma}\cite{Urakawa}
Let $(M,g_{M})$ and $(N,g_{N})$ be Riemannian manifolds and suppose that $%
\psi:M \longrightarrow N$ is a smooth map between them. Then we have
\begin{equation}
\nabla^{\psi}_{\xi}{\rm d}\psi(\eta)-\nabla^{\psi}_{\eta}{\rm d}\psi(\xi)-{\rm d}\psi([\xi,\eta])=0 \label{U}
\end{equation}
for $\xi,\eta\in\Gamma(TM)$.
\end{lemma}
\begin{remark}
From (\ref{U}), one can easily see that if $\xi$ is basic and 
$\eta\in\Gamma(({\rm kerd}\psi)^\perp)$, then $[\xi,\eta]\in\Gamma({\rm kerd}\psi).$
\end{remark}
Finally, we have the following from \cite{BW}
\begin{lemma}
\label{lem1}(Second fundamental form of an HC submersion) Suppose that $%
\psi:M_1\rightarrow M_2$ is a horizontally conformal submersion. Then, for
any horizontal vector fields $\xi,\eta$ and vertical vector fields $V,W,$ we have 
\begin{enumerate}
\item[(i)]$ (\nabla {\rm d}\psi)(\xi,\eta)=\xi(\ln\lambda){\rm d}\psi\eta+\eta(\ln\lambda){\rm d}\psi\xi-g(\xi,\eta){\rm d}\psi(\nabla\ln\lambda); $
\item[(ii)]$ (\nabla {\rm d}\psi)(V,W)=-{\rm d}\psi(\mathcal{T}_{V}W); $
\item[(iii)]$ (\nabla {\rm d}\psi)(\xi,V)=-{\rm d}\psi(\nabla_{\xi}^{1}V)=-{\rm d}\psi(\mathcal{A}_{\xi}V).$
\end{enumerate}
\end{lemma}

\section{\textbf{Conformal generic submersions from almost Hermitian manifolds}}

In this section, we define conformal generic submersions from an almost
Hermitian manifold onto a Riemannian manifold, give lots of examples and investigate the geometry of leaves
of distributions and show that there are certain product structures on the total
space of a conformal generic submersion. 

Let $(M_1,g_1,J)$ be an almost Hermitian manifold with almost complex structure $J$ and a Riemannian
metric $g$ such that \cite{YK3}
\begin{equation}
(i) \ J^{2}=-I,\mbox{ } \ \ \ \ (ii) \ g(Z_1,Z_2)=g(JZ_1,JZ_2),  \label{J}
\end{equation}
for all vector fields $Z_1$, $Z_2$ on $M_1,$ where $I$ is the identity map.
An almost Hermitian manifold $M_1$ is called K\"{a}hler manifold if the almost complex structure $J$ satisfies
\begin{equation}
(\nabla_{Z_1}J)Z_2=0,\mbox{ }\forall Z_1,Z_2\in\Gamma(TM_1),
\label{e.q:2.2}
\end{equation}
where $\nabla$ denotes the Levi-Civita connection on $M_1$.

First of all, we recall the definition of generic Riemannian submersions as follows:

\begin{definition}\cite{Shahid}
Let $N_1$ be a complex $m$-dimensional almost Hermitian manifold
with Hermitian metric $h_1$ and almost complex structure $J_1$ and $N_2$ be
a Riemannian manifold with Riemannian metric $h_{2}.$ A Riemannian 
submersion $\psi:N_1 \longrightarrow N_2$ is called 
generic Riemannian submersion if there is a distribution $\mathcal{\bar{D}}_{1}\subseteq \ker d\psi$ such that
\begin{equation*}
ker d\psi=\mathcal{\bar{D}}_{1}\oplus \mathcal{\bar{D}}_{2} \ \ J(\mathcal{\bar{D}}_{1})=\mathcal{\bar{D}}_{1},
\end{equation*}
where $\mathcal{\bar{D}}_{2}$ is orthogonal complementary to $\mathcal{\bar{D}}_{1}$ in $(ker d\psi),$ and
is purely real distribution on the fibres of the submersion $\psi.$
\end{definition}
Now, we will give our definition as follows:

Let $\phi$ be a conformal submersion from an almost Hermitian manifold $(M,g,J) $ to a Riemannian manifold
$(B,h).$ Define $$\mathcal{D}_q=({\rm kerd}\phi \cap J({\rm kerd}\phi)), \ \ q\in M$$ 
the complex subspace of the vertical subspace $\mathcal{V}_q.$
\begin{definition}
\label{def} Let $\phi:(M,g,J) \longrightarrow (B,h)$ be a horizontally conformal submersion, where 
$(M,g,J)$ is an almost Hermitian manifold and $(B,h)$ is
a Riemannian manifold with Riemannian metric $h.$ If the dimension $\mathcal{D}_q$ is constant along $M$
and it defines a differential distribution on $M$ then we say that $\phi$ is conformal generic submersion.
A conformal generic submersion is purely real (respectively, complex) if  $\mathcal{D}_q=\{0\}$ (respectively, $\mathcal{D}_q={\rm kerd}\phi_q$). 
For a conformal generic submersion, the orthogonal complementary distribution $\mathcal{D}^\perp,$ called purely real distribution, satisfies
 \begin{equation}
{\rm kerd}\phi=\mathcal{D}\oplus \mathcal{D}^\perp,  \label{dphi}
\end{equation}
and
\begin{equation}
\mathcal{D}\cap \mathcal{D}^\perp=\{0\}. \label{dphi1}
\end{equation}
\end{definition}
\begin{remark}
It is known that the distribution ${\rm kerd}\phi$ is integrable.
Hence, above definition implies that the integral manifold (fiber) $%
\phi^{-1}(q) $, $q\in B$, of $kerd\phi$ is a generic submanifold of $M.$
For generic submanifolds, see:\cite{Chen}.
\end{remark}

First of all, we give lots of examples for conformal generic submersions
 from almost Hermitian manifolds to Riemannian manifolds.

\begin{example}
Every semi-invariant Riemannian submersion \cite{S2} $\phi$ from an almost Hermitian
manifold to a Riemannian manifold is a conformal generic submersion with $%
\lambda=1$ and $\mathcal{D}^\perp$ is a totally real distribution.
\end{example}

\begin{example}
Every slant submersion \cite{S3}  $\phi$ from an almost Hermitian
manifold to a Riemannian manifold is a conformal generic submersion such that 
$\lambda=1,$ $\mathcal{D}=\{0\}$ and $\mathcal{D}^\perp$ is a slant distribution.
\end{example}

\begin{example}
 Every semi-slant submersion  \cite{PP} $\phi$ from an almost Hermitian
manifold to a Riemannian manifold is a conformal generic submersion such that 
$\lambda=1$ and $\mathcal{D}^\perp$ is a slant distribution.
\end{example}

\begin{example}
 Every conformal semi-invariant submersion \cite{As} $\phi$ from an almost Hermitian
manifold to a Riemannian manifold is a conformal generic submersion such that 
$\mathcal{D}^\perp$ is a totally real distribution.
\end{example}

\begin{example}
 Every generic Riemannian submersion \cite{Shahid} $\phi$ from an almost Hermitian
manifold to a Riemannian manifold is a conformal generic submersion with $\lambda=1.$
\end{example}

\begin{remark}
We would like to point out that since 
conformal semi-invariant submersions include conformal holomorphic submersions and conformal anti-invariant submersions,
such conformal submersions are also examples of conformal generic submersions. 
We say that a conformal generic submersion is proper if $\lambda\neq1$ and $\mathcal{D}^\perp$ is neither complex nor purely real.
\end{remark}

In the following $\mathbb{R}^{2m}$ denotes the Euclidean
$2m$-space with the standard metric. Define the compatible almost complex structure $J$
 on $\mathbb{R}^{8}$ by%
\begin{align*}
&J\partial_1=\frac{1}{\sqrt{2}}(-\partial_3-\partial_2), J\partial_2=\frac{1}{\sqrt{2}}(-\partial_4+\partial_1), 
J\partial_3=\frac{1}{\sqrt{2}}(\partial_1+\partial_4), J\partial_4=\frac{1}{\sqrt{2}}(\partial_2-\partial_3),\\
&J\partial_5=\frac{1}{\sqrt{2}}(-\partial_7-\partial_6), J\partial_6=\frac{1}{\sqrt{2}}(-\partial_8+\partial_5), 
J\partial_7=\frac{1}{\sqrt{2}}(\partial_5+\partial_8), J\partial_8=\frac{1}{\sqrt{2}}(\partial_6-\partial_7)
\end{align*}
where $\partial_k=\frac{\partial}{\partial u_k},\ \ k=1,...,8$ and $(u_1,...,u_8)$ natural coordinates of $\mathbb{R}^{8}.$
\begin{example}
\label{exm1}Let $\phi:(\mathbb{R}^{8},g) \longrightarrow (\mathbb{R}^{2},h)$ be a submersion defined by
\begin{equation*}
\phi(u_{1},u_{2},...,u_{8})=(t_1,t_2),
\end{equation*}
where
$$t_1=e^{u_{1}}\sin u_{3}\ \  \textrm{and}\ \  t_2=e^{u_{1}}\cos u_{3}.$$ Then, the Jacobian matrix of $\phi$ is:
\[
d\phi=
\begin{bmatrix}
    e^{u_{1}}\sin u_{3}  & 0 & e^{u_{1}}\cos u_{3} & 0 & 0 & 0 & 0 & 0\\
    e^{u_{1}}\cos u_{3}  & 0 & -e^{u_{1}}\sin u_{3} & 0 & 0 & 0 & 0 & 0 \\
  \end{bmatrix}
.
\]
A straight computations yields
\begin{equation*}
kerd\phi=span\{T_{1}=\partial_{2},\ T_{2}=\partial_{4},\
T_{3}=\partial_{5},\ T_{4}=\partial_{6},\ T_{5}=\partial_{7},\
T_{6}=\partial_{8}\}
\end{equation*}%
and
\begin{align*}
(kerd\phi)^{\perp }\!=\!span\{& H_{1}\!=\!e^{u_{1}}\sin {u_{3}}\partial_{1}+e^{u_{1}}\cos {u_{3}}\partial_{3},
H_{2}\!=\!e^{u_{1}}\cos {u_{3}}\partial_{1}-e^{u_{1}}\sin {u_{3}}\partial_{3}\}.
\end{align*}%
Hence we get
$$JT_{3}\!=\!\frac{1}{\sqrt{2}}T_{4}+\frac{1}{\sqrt{2}}T_{5},\,\, JT_{4}\!=\!-\frac{1}{\sqrt{2}}T_{3}+\frac{1}{\sqrt{2}}T_{6},$$ $$JT_{5}\!=\!-\frac{1}{\sqrt{2}}T_{3}-\frac{1}{\sqrt{2}}T_{6},\,\,JT_{6}\!=\!-\frac{1}{\sqrt{2}}T_{4}+\frac{1}{\sqrt{2}}T_{5}$$
and
$$ JT_{1}=\frac{1}{\sqrt{2}}T_{2}-\frac{e^{-u_{1}}\sin u_{3}}{
\sqrt{2}}H_{1}-\frac{e^{-u_{1}}\cos u_{3}}{\sqrt{2}}H_{2},$$
$$JT_{2}=-\frac{1}{\sqrt{2}}T_{1}+\frac{e^{-u_{1}}\cos u_{3}}{\sqrt{2}}H_{1}-\frac{%
e^{-u_{1}}\sin u_{3}}{\sqrt{2}}H_{2},$$ where $J$ is the complex structure of $\mathbb{R}^8.$
It follows that $\mathcal{D}=span\{T_{3},T_{4},T_{5},T_{6}\}$ and $\mathcal{D}^\perp=span\{T_{1},T_{2}\}$. Also
by direct computations yields
$$d\phi(H_{1})=(e^{u_{1}})^{2}\partial v_{1} \ \  \ \textrm{and} \ \ \  d\phi(H_{2})=(e^{u_{1}})^{2}\partial v_{2}.$$
Hence, it is easy to see that
\begin{equation*}
g_{\mathbb{R}^2}(d\phi(H_{i}),d\phi(H_{i}))=(e^{u_{1}})^{2}g_{\mathbb{R}^8}(H_{i},H_{i}),\ \ i=1,2.
\end{equation*}%
Thus $\phi $ is a conformal generic submersion with $\lambda =e^{u_{1}}.$
\end{example}

\begin{example}
\label{exm2}Let $\psi:(\mathbb{R}^{8},g_1) \longrightarrow (\mathbb{R}^{2},g_2)$ be a submersion defined by
\begin{equation*}
\psi(v_{1},v_{2},...,v_{8})=\pi^{17}(\frac{-v_1+v_3}{\sqrt{2}},\frac{-v_1-v_3}{\sqrt{2}}).
\end{equation*}
Then $\psi$ is a conformal generic submersion with $\lambda =\pi^{17}.$

\end{example}

\begin{remark}
Throughout this paper, we assume that all horizontal vector fields are basic vector fields.
\end{remark}
Let $\phi$ be a conformal generic submersion from a K\"{a}hler manifold $%
(M,g,J)$ onto a Riemannian manifold $(B,h)$. Then for $%
Z\in\Gamma({\rm kerd}\phi)$, we write
\begin{equation}
JZ=\varphi Z+\omega Z  \label{jv}
\end{equation}
where $\varphi Z\in\Gamma({\rm kerd}\phi)$ and $\omega Z\in\Gamma(({\rm kerd}\phi)^\perp)$.
We denote the orthogonal complement of $\omega \mathcal{D}^\perp$ in $({\rm kerd}\phi)^\perp$ by $\mu.$ Then we have
\begin{equation}
({\rm kerd}\phi)^\perp=\omega \mathcal{D}^\perp \oplus \mu  \label{y}
\end{equation}
and that $\mu$ is invariant under $J.$ Also for $\xi\in\Gamma(({\rm kerd}\phi)^\perp)$, we write
\begin{equation}
J\xi=\mathcal{B}\xi+\mathcal{C}\xi  \label{jx}
\end{equation}
where $\mathcal{B}\xi\in\Gamma(\mathcal{D}^\perp)$ and $\mathcal{C}\xi\in\Gamma(\mu)$. From (\ref{jv}), (\ref{y}) and (\ref{jx}) we have the following 
\begin{proposition}
Let $\phi$ be a conformal generic submersion from a K\"{a}hler manifold $(M,g,J)$ onto a Riemannian manifold $(B,h)$. Then we have 
\begin{align*}
&(i) \ \varphi \mathcal{D}=\mathcal{D},\, \, \, \, \, (ii) \ \omega \mathcal{D}=0,\,\, \, \, \, (iii) \ \varphi \mathcal{D}^\perp\subset \mathcal{D}^\perp,\,\, \, \, \, 
(d) \ \mathcal{B}(({\rm kerd}\phi)^\perp)=\mathcal{D}^\perp, \\
&(a) \ \varphi^2+\mathcal{B}\omega=-id,\,\,(b) \ \mathcal{C}^2+\omega\mathcal{B}=-id,\,\, (c) \ \omega\varphi+\mathcal{C}\omega=0,\,\,
(d) \ \mathcal{B}\mathcal{C}+\varphi\mathcal{B}=0.
\end{align*}
\end{proposition}
Next, we easily have the following lemma:
\begin{lemma}
Let $(M,g,J)$ be a K\"{a}hler manifold and $(B,h)$ a Riemannian
mani-fold. Let $\phi:(M,g,J)\rightarrow (B,h)$ be a conformal generic
submersion. Then we have

\begin{enumerate}
\item[(i)]
\begin{align*}
\mathcal{C}\mathcal{H}\nabla^{1}_\xi\eta+\omega \mathcal{A}_{\xi}\eta&=\mathcal{A}_{\xi}\mathcal{B}\eta+\mathcal{H}\nabla^{1}_{\xi}\mathcal{C}\eta \\
\mathcal{B}\mathcal{V}\nabla^{1}_\xi\eta+\varphi \mathcal{A}_{\xi}\eta&=\mathcal{V}\nabla^{1}_{\xi}\mathcal{B}\eta+\mathcal{A}_{\xi}\mathcal{C}\eta,
\end{align*}
\item[(ii)]
\begin{align*}
\mathcal{C}\mathcal{T}_{Z}W+\omega\hat{\nabla}_{Z}W&=\mathcal{T}_{U}\varphi V+\mathcal{A}_{\omega W}Z \\
\mathcal{B}\mathcal{T}_{Z}W+\varphi\hat{\nabla}_{Z}W&=\hat{\nabla}_{Z}\varphi W+\mathcal{T}_{Z}\omega W,
\end{align*}
\item[(iii)]
\begin{align*}
\mathcal{C}\mathcal{A}_{\xi}Z+\omega\mathcal{V}\nabla^{1}_\xi Z &=\mathcal{A}_{\xi}\varphi Z+\mathcal{H}\nabla^{1}_{Z}\omega W\\
\mathcal{B}\mathcal{A}_{\xi}Z+\varphi\mathcal{V}\nabla^{1}_\xi Z&=\mathcal{V}\nabla^{1}_{\xi}\varphi Z+\mathcal{A}_{\xi}\omega Z,
\end{align*}
\end{enumerate}
for $\xi,\eta\in\Gamma(({\rm kerd}\phi)^\perp)$ and $Z,W\in\Gamma({\rm kerd}\phi)$.
\end{lemma}

\subsection{\textbf{The geometry of $\phi:(M,g,J)\rightarrow (B,h)$}}

\begin{lemma}
\label{lem1} Let $\phi$ be a conformal generic submersion from a K\"{a}hler
manifold $(M,g,J)$ onto a Riemannian manifold $(B,h)$. Then the
distribution $\mathcal{D}$ is integrable if and only if the following is satisfied
\begin{align}
{\lambda^{-2}}\{h((\nabla {\rm d}\phi)(U,JV)-(\nabla{\rm d}\phi)(V,JU),{\rm d}\phi(\omega Z)\}=g(\varphi(\hat{\nabla}_{V}JU-\hat{\nabla}_{U}JV),Z)  \label{dintg}
\end{align}
for $U,V\in\Gamma(\mathcal{D})$ and $Z\in\Gamma(\mathcal{D}^\perp).$
\end{lemma}
\begin{proof}
The distribution $\mathcal{D}$ is integrable if and only if 
\begin{equation*}
g([U,V],Z)=0,\ \ \textrm{and} \ \ g([U,V],\xi)=0
\end{equation*}
 for any $U,V\in\Gamma(\mathcal{D}),\; Z\in\Gamma(\mathcal{D}^\perp)$
and $\xi\in\Gamma(({\rm ker d}\phi)^\perp)$. Since ${\rm ker d}\phi$ is integrable $g([U,V],\xi)=0$.
Therefore, $\mathcal{D}$ is integrable if and only if $g([U,V],Z)=0.$
By Eq. (\ref{J})(i), Eq. (\ref{e.q:2.2}), Eq. (\ref{nvw}) and Eq. (\ref{jv}) we have
\begin{align*}
g([U,V],Z)=g(\mathcal{H}\nabla^{^{M}}_{U}JV,\omega Z)+g(\hat{%
\nabla}_{U}JV,\varphi Z)\!\!-\!\!g(\mathcal{H}\nabla^{^{M}}_{V}JU,\omega
Z)\!\!-\!\!g(\hat{\nabla}_{V}JU,\varphi Z).
\end{align*}
By using the property of $\phi$, Eq. (\ref{jv}) and Lemma \ref{lem1} yields
\begin{align*}
g([U,V],Z)&={\lambda^{-2}}h\big(-(\nabla{\rm d}\phi)(U,JV)+\nabla^{\phi}_{U}{\rm d}\phi(JV),{\rm d}\phi(\omega Z)\big)-{\lambda^{-2}}h\big(-(\nabla%
{\rm d}\phi)(V,JU) \\
& +\nabla^{\phi}_{V}{\rm d}\phi(JU),{\rm d}\phi(\omega Z)\big)+g(\varphi(\hat{\nabla}_{V}JU-\hat{\nabla}_{U}JV),Z) \\
&=\lambda^{-2}\{h((\nabla{\rm d}\phi)(V,JU)\!\!-\!\!(\nabla{\rm d}\phi)(U,JV),{\rm d}\phi(\omega Z))\}\!\!+\!\!g(\varphi(\hat{\nabla}_{V}JU\!\!-\!\!\hat{\nabla}_{U}JV),Z)
\end{align*}
which gives Eq. (\ref{dintg}).\newline

In a similar way, we get:
\end{proof}

\begin{lemma}
\label{lem2} Let $\phi$ be a conformal generic submersion from a K\"{a}hler
manifold $(M,g,J)$ onto a Riemannian manifold $(B,h)$. Then the
distribution $\mathcal{D}^\perp$ is integrable if and only if
\begin{align}
\hat{\nabla}_{V}\varphi U-\hat{\nabla}_{U}\varphi V+\mathcal{T}_{V}\omega U-\mathcal{T}_{U}\omega
V\in\Gamma(\mathcal{D}^\perp)  \label{d1intg}
\end{align}
for $U,V\in\Gamma(\mathcal{D}^\perp).$
\end{lemma}

We now investigate the geometry of leaves of $\mathcal{D}$
and $\mathcal{D}^\perp$.

\begin{lemma}
\label{lem3} Let $\phi$ be a conformal generic submersion from a K\"{a}hler
manifold $(M,g,J)$ to a Riemannian manifold $(B,h)$. Then $\mathcal{D}$
defines a totally geodesic foliation on $M$ if and only if

\begin{enumerate}
\item[(a)] $\lambda^{-2}h((\nabla {\rm d}\phi)(X_{1},JY_{1}),{\rm d}\phi(\omega
X_2))=g(\hat{\nabla}_{X_{1}}JY_{1},\varphi X_2)$
\item[(b)]$\begin{aligned}[t]
\lambda^{-2}h((\nabla {\rm d}\phi)(X_{1},JY_{1}),{\rm d}\phi(\mathcal{C}%
\xi))=g(\hat{\nabla}_{X_{1}}\varphi\mathcal{B}X+\mathcal{T}_{X_1}\omega\mathcal{B}X,Y_1)
\end{aligned}$
\end{enumerate}
for $X_{1},Y_{1}\in\Gamma(\mathcal{D}),X_{2}\in\Gamma(\mathcal{D}^\perp)$ and $%
\xi\in\Gamma(({\rm kerd}\phi)^\perp)$.
\end{lemma}
\begin{proof}
The distribution $\mathcal{D}$ defines a totally geodesic foliation on $M$ if and
only if 
\begin{equation*}
g(\nabla^{^{M}}_{X_{1}}Y_{1},X_{2})=0\ \ \textrm{and} \ \ g(\nabla^{^{M}}_{X_{1}}Y_{1},\xi)=0
\end{equation*}
for any $X_{1},Y_{1}\in\Gamma(\mathcal{D}),\,X_{2}\in\Gamma(\mathcal{D}^\perp)$ and $\xi\in\Gamma(({\rm kerd}\phi)^\perp)$.
 By virtue of Eq. (\ref{J})(i) and Eq. (\ref{J})(ii), we get
\begin{equation*}
g(\nabla^{^{M}}_{X_{1}}Y_{1},X_{2})=g(\hat{\nabla}_{X_1}JY_1,\varphi
X_2)+g(\mathcal{H}\nabla^{^{M}}_{X_{1}}JY_{1},\omega X_{2}).
\end{equation*}
Since $\phi$ is a conformal generic submersion, by Eq. (\ref{nfixy}) yields
\begin{align}  \label{nx1y1x2}
g(\nabla^{1}_{X_{1}}Y_{1},X_{2})&=g(\hat{\nabla}_{X_1}JY_1,\varphi
X_2)-\lambda^{-2}h((\nabla {\rm d}\phi)(X_{1},JY_{1}),{\rm d}\phi(\omega
X_{2})).
\end{align}
On the other hand, by Eq. (\ref{J})(i), Eq. (\ref{J})(ii), Eq. (\ref{nvw}) and Eq. (\ref{jx}) yields
\begin{equation*}
g(\nabla^{1}_{X_{1}}Y_{1},\xi)=g(Y_1,\nabla^{1}_{X_1}J%
\mathcal{B}\xi)+g(\mathcal{H}\nabla^{1}_{X_{1}}JY_{1},\mathcal{C}\xi).
\end{equation*}
By Eq. (\ref{nvx}), Eq. (\ref{nfixy}) and Eq. (\ref{jv}) we get
\begin{align}  \label{x1y1x}
g(\nabla^{1}_{X_{1}}Y_{1},\xi)&=g(Y_1,\hat{\nabla}_{X_{1}}\varphi\mathcal{B}\xi)+g(Y_1,\mathcal{T}_{X_1}\omega\mathcal{B}\xi)\\
&-\lambda^{-2}h((\nabla {\rm d}\phi)(X_{1},JY_{1}),{\rm d}\phi(\mathcal{C}\xi)).\notag
\end{align}
Hence proof follows from Eq. (\ref{nx1y1x2}) and Eq. (\ref{x1y1x}).
\end{proof}

In a similar way, we have the following result.

\begin{lemma}
\label{lem4} Let $\phi$ be a conformal generic submersion from a K\"{a}hler
manifold $(M,g,J)$ to a Riemannian manifold $(B,h)$. Then $\mathcal{D}^\perp$
defines a totally geodesic foliation on $M$ if and only if
\begin{enumerate}
\item[(a)] $\lambda^{-2}h((\nabla {\rm d}\phi)(X_{2},JX_{1}),{\rm d}\phi(\omega Y_2))=g(\hat{\nabla}_{X_{2}}JX_{1},\varphi Y_2)$,
\item[(b)]$\begin{aligned}[t]
\lambda^{-2}h\big((\nabla {\rm d}\phi)(X_2,Y_2),{\rm d}\phi(JC\xi)\big)=g(\mathcal{T}_{X_2}\mathcal{B}\xi,\omega Y_2)-g(\hat{\nabla}_{X_2}\varphi Y_2,\mathcal{B}\xi)
\end{aligned}$
\end{enumerate}
for $X_{1},Y_{1}\in\Gamma(\mathcal{D}),X_{2},Y_{2}\in\Gamma(\mathcal{D}^\perp)$ and $%
\xi\in\Gamma(({\rm kerd}\phi)^\perp)$.
\end{lemma}
From Lemma \ref{lem3} and Lemma \ref{lem4}, we have the following result.

\begin{lemma}\label{lemma} 
Let $\phi:(M,g,J)\longrightarrow (B,h)$ be a
conformal generic submersion from a K\"{a}hler manifold $(M,g,J)$ onto
a Riemannian manifold $(B,h)$. Then the fibers of $\phi$ are locally
product manifold of the form $M_{\mathcal{D}}\times M_{\mathcal{D}^\perp}$ if and only if
\begin{enumerate}
\item[(i)] $\lambda^{-2}h((\nabla {\rm d}\phi)(X_{1},JY_{1}),{\rm d}\phi(\omega
X_2))=g(\hat{\nabla}_{X_{1}}JY_{1},\varphi X_2)$
\item[(ii)]$\begin{aligned}[t]
\lambda^{-2}h((\nabla {\rm d}\phi)(X_{1},JY_{1}),{\rm d}\phi(\mathcal{C}%
\xi))=g(\hat{\nabla}_{X_{1}}\varphi\mathcal{B}\xi+\mathcal{T}_{X_1}\omega\mathcal{B}\xi,Y_1)
\end{aligned}$
\item[(iii)] $\lambda^{-2}h((\nabla {\rm d}\phi)(X_{2},JX_{1}),{\rm d}\phi(\omega Y_2))=g(\hat{\nabla}_{X_{2}}JX_{1},\varphi Y_2)$,
\item[(iv)]$\begin{aligned}[t]
\lambda^{-2}h\big((\nabla {\rm d}\phi)(X_2,Y_2),{\rm d}\phi(JC\xi)\big)=g(\mathcal{T}_{X_2}\mathcal{B}\xi,\omega Y_2)
-g(\hat{\nabla}_{X_2}\varphi Y_2,\mathcal{B}\xi)
\end{aligned}$
\end{enumerate}
for any $X_{1},Y_{1}\in\Gamma(\mathcal{D}),X_{2},Y_2\in\Gamma(\mathcal{D}^\perp)$ and $\xi\in\Gamma(({\rm kerd}\phi)^\perp).$
\end{lemma}

Since the distribution ${\rm kerd}\phi$ is integrable, we only study the
integrability of the distribution $({\rm kerd}\phi)^\perp$ and then we discuss
the geometry of leaves of ${\rm kerd}\phi$ and $({\rm kerd}\phi)^\perp$.

\begin{theorem}
\label{dinteg} Let $\phi$ be a conformal generic submersion from a K\"{a}hler
manifold $(M_1,g_{1},J)$ to a Riemannian manifold $(M_2,g_{2})$. Then the following conditions are equivalent:
\begin{enumerate}
\item[(i)] The distribution $({\rm kerd}\phi)^\perp$ is integrale.
\item[(ii)] $\mathcal{V}(\nabla^{1}_{\xi}\mathcal{B}\eta-\nabla^{1}_{\eta}\mathcal{B}\xi)+\mathcal{A}_\xi{\mathcal{C}\eta}
-\mathcal{A}_\eta{\mathcal{C}\xi}\in\Gamma(\mathcal{D}^\perp).$
\item[(iii)]
$\begin{aligned}[t]
&\lambda^{-2}h\big((\nabla{\rm d}\phi)(\xi,\mathcal{B}\eta)-(\nabla{\rm d}\phi)(\eta,\mathcal{B}\xi)-\nabla^{\phi}_{\xi}{\rm d}\phi(\mathcal{C}\eta)
+\nabla^{\phi}_{\eta}{\rm d}\phi(\mathcal{C}\xi),{\rm d}\phi(\omega W)\big)\\
&=g\big(\eta(\ln\lambda)\mathcal{C}\xi\!-\!\xi(\ln\lambda)\mathcal{C}\eta\!-\!\mathcal{C}\eta(\ln\lambda)\xi\!\!+\!\!\mathcal{C}\xi(\ln\lambda)\eta
\!\!+\!\!2g(\xi,\mathcal{C}\eta)\nabla\ln\lambda,\omega W\big) \\
&+g(-\varphi(\mathcal{V}\nabla^{1}_{\xi}\mathcal{B}\eta-\mathcal{V}\nabla^{1}_{\eta}\mathcal{B}\xi+\mathcal{A}_{\xi}\mathcal{C}\eta-\mathcal{A}_{\eta}\mathcal{C}\xi),W)\notag 
\end{aligned}$
\end{enumerate}
for $\xi,\eta\in\Gamma({\rm kerd}\phi)^\perp),$  $V\in\Gamma(\mathcal{D})$ and $W\in\Gamma(\mathcal{D}^\perp)$.
\end{theorem}

\begin{proof}
By virtue of (\ref{J})(i) and (\ref{J})(ii), we get
\begin{equation*}
g([\xi,\eta],JV)=g(-J[\xi,\eta],V)=-g(J\nabla^{1}_{\xi}J\eta,JV_1)+g(J\nabla^{1}_{\eta}J\xi,JV)
\end{equation*}
for $\xi,\eta\in\Gamma(({\rm kerd}\phi)^\perp)$ and $V_1\in\Gamma(\mathcal{D}).$ Then by using (\ref{jx}), (\ref{jv}), (\ref{nxv}) and (\ref{nxy}) yields
\begin{align*}
g([\xi,\eta],JV)&=-g(\phi(\mathcal{V}(\nabla^{1}_{\xi}\mathcal{B}\eta-\nabla^{1}_{\eta}\mathcal{B}\xi)+\mathcal{A}_\xi{\mathcal{C}\eta}-\mathcal{A}_\eta{\mathcal{C}\xi}),JV)
\end{align*}
So that
\begin{align}\label{yint1}
g([\xi,\eta],JV)=0\Longleftrightarrow \mathcal{V}(\nabla^{1}_{\xi}\mathcal{B}\eta-\nabla^{1}_{\eta}\mathcal{B}\xi)
+\mathcal{A}_\xi{\mathcal{C}\eta}-\mathcal{A}_\eta{\mathcal{C}\xi}\in\Gamma(\mathcal{D}^\perp).
\end{align}
Also using (\ref{nxv}), (\ref{nxy}) and (\ref{jx}) we get
\begin{align*}
g([\xi,\eta],W)&=g(\mathcal{V}\nabla^{1}_{\xi}\mathcal{B}\eta-\mathcal{V}%
\nabla^{1}_{\eta}\mathcal{B}\xi+\mathcal{A}_{\xi}\mathcal{C}\eta-\mathcal{A}_{\eta}\mathcal{C}\xi,\varphi W)
\!\!+\!\!g(\mathcal{H}\nabla^{1}_{\xi}\mathcal{B}\eta,\omega W) \\
&-g(\mathcal{H}\nabla^{1}_{\eta}\mathcal{B}\xi,\omega W)+g_{1}(\mathcal{H%
}\nabla^{1}_{\xi}\mathcal{C}\eta,\omega W)-g(\mathcal{H}\nabla^{1}_{\eta}\mathcal{C}\xi,\omega W).
\end{align*}
Taking into account (\ref{nfixy}) and Lemma \ref{lem1}, we get
\begin{align*}
g([\xi,&\eta],W)=g(\mathcal{V}\nabla^{1}_{\xi}\mathcal{B}\eta-\mathcal{V}\nabla^{1}_{\eta}\mathcal{B}\xi+\mathcal{A}_{\xi}\mathcal{C}\eta-\mathcal{A}_{\eta}\mathcal{C}\xi,\varphi W)\\
&\!\!-\!\!\lambda^{-2}h((\nabla{\rm d}\phi)(\xi,\mathcal{B}\eta),{\rm d}\phi(\omega W))\!\!+\!\!\lambda^{-2}h((\nabla{\rm d}\phi(\eta,\mathcal{B}\xi),{\rm d}\phi(\omega W)) \\
&\!\!+\!\!\lambda^{-2}h\{\!-\!\xi(\ln\lambda){\rm d}\phi(\mathcal{C}\eta)\!\!-\!\!\mathcal{C}\eta(\ln\lambda){\rm d}\phi(\xi)\!\!+\!\!g(\xi,\mathcal{C}\eta){\rm d}\phi(\nabla\ln\lambda)
\!\!+\!\!\nabla^{\phi}_{\xi}{\rm d}\phi(\mathcal{C}\eta),{\rm d}\phi(\omega W)\} \\
&\!\!-\!\!\ \lambda^{-2}h\{\!-\!\eta(\ln\lambda){\rm d}\phi(\mathcal{C}\xi)\!\!-\!\!\mathcal{C}\xi(\ln\lambda){\rm d}\phi(\eta)\!\!+\!\!g(\eta,\mathcal{C}\xi){\rm d}\phi(\nabla\ln\lambda)
\!\!+\!\!\nabla^{\phi}_{\eta}{\rm d}\phi(\mathcal{C}\xi),{\rm d}\phi(\omega W)\}
\end{align*}
by virtue of (\ref{nxv}) and (\ref{nfixy}) 
\begin{align*}
g([\xi,\eta],&W)=g(\mathcal{V}\nabla^{1}_{\xi}\mathcal{B}\eta-\mathcal{V}
\nabla^{1}_{\eta}\mathcal{B}\xi+\mathcal{A}_{\xi}\mathcal{C}\eta-\mathcal{A}_{\eta}\mathcal{C}\xi,\varphi W)\\
&-\lambda^{-2}h\big((\nabla{\rm d}\phi)(\xi,\mathcal{B}\eta)-(\nabla{\rm d}\phi)(\eta,\mathcal{B}\xi),{\rm d}\phi(\omega W)\big)\\
&-\lambda^{-2}g(\nabla\ln\lambda,\xi)h({\rm d}\phi(\mathcal{C}\eta),{\rm d}\phi(\omega W)) 
-\lambda^{-2}g(\nabla\ln\lambda,\mathcal{C}\eta)h({\rm d}\phi(\xi),{\rm d}\phi(\omega W)) \\
&+\lambda^{-2}g(\xi,\mathcal{C}\eta)h({\rm d}\phi(\nabla\ln\lambda),{\rm d}\phi(\omega W))
+\lambda^{-2}h(\nabla^{\phi}_{\xi}{\rm d}\phi(\mathcal{C}\eta),{\rm d}\phi(\omega W)) \\
&+\lambda^{-2}g(\nabla\ln\lambda,\eta)h({\rm d}\phi(\mathcal{C}\xi),{\rm d}\phi(\omega W)) 
+\lambda^{-2}g(\nabla\ln\lambda,\mathcal{C}\xi)h({\rm d}\phi(\eta),{\rm d}\phi(\omega W)) \\
&-\lambda^{-2}g(\eta,\mathcal{C}\xi)h({\rm d}\phi(\nabla\ln\lambda),{\rm d}\phi(\omega W)) 
-\lambda^{-2}h(\nabla^{\phi}_{\eta}{\rm d}\phi(\mathcal{C}\xi),{\rm d}\phi(\omega W))
\end{align*}
A straight computation yields
\begin{align*} 
g([\xi,&\eta],W)\!\!=\!\!g\big(\eta(\ln\lambda)\mathcal{C}\xi\!-\!\xi(\ln\lambda)\mathcal{C}\eta\!-\!\mathcal{C}\eta(\ln\lambda)\xi\!\!+\!\!\mathcal{C}\xi(\ln\lambda)\eta
\!\!+\!\!2g(\xi,\mathcal{C}\eta)\nabla\ln\lambda,\omega W\big) \\
&+g(-\varphi(\mathcal{V}\nabla^{1}_{\xi}\mathcal{B}\eta-\mathcal{V}\nabla^{1}_{\eta}\mathcal{B}\xi+\mathcal{A}_{\xi}\mathcal{C}\eta-\mathcal{A}_{\eta}\mathcal{C}\xi),W)\notag \\
&-\lambda^{-2}h\big((\nabla{\rm d}\phi)(\xi,\mathcal{B}\eta)-(\nabla{\rm d}\phi)(\eta,\mathcal{B}\xi)-\nabla^{\phi}_{\xi}{\rm d}\phi(\mathcal{C}\eta)
+\nabla^{\phi}_{\eta}{\rm d}\phi(\mathcal{C}\xi),{\rm d}\phi(\omega W)\big)\notag
\end{align*}
so that 
\begin{align} \label{yint2}
g([\xi,\eta],W)=&0 \Longleftrightarrow \lambda^{-2}h\big((\nabla{\rm d}\phi)(\xi,\mathcal{B}\eta)-(\nabla{\rm d}\phi)(\eta,\mathcal{B}\xi)\\
&-\nabla^{\phi}_{\xi}{\rm d}\phi(\mathcal{C}\eta)+\nabla^{\phi}_{\eta}{\rm d}\phi(\mathcal{C}\xi),{\rm d}\phi(\omega W)\big)\notag\\
&=g\big(\eta(\ln\lambda)\mathcal{C}\xi\!-\!\xi(\ln\lambda)\mathcal{C}\eta\!-\!\mathcal{C}\eta(\ln\lambda)\xi\!\!+\!\!\mathcal{C}\xi(\ln\lambda)\eta
\!\!+\!\!2g(\xi,\mathcal{C}\eta)\nabla\ln\lambda,\omega W\big) \notag\\
&+g(-\varphi(\mathcal{V}\nabla^{1}_{\xi}\mathcal{B}\eta-\mathcal{V}\nabla^{1}_{\eta}\mathcal{B}\xi+\mathcal{A}_{\xi}\mathcal{C}\eta-\mathcal{A}_{\eta}\mathcal{C}\xi),W)\notag 
\end{align}
By using (\ref{yint1}) and (\ref{yint2}) we obtain  $(i) \Longleftrightarrow (ii), (i) \Longleftrightarrow (iii)$ which completes the proof.
\end{proof}

From Theorem \ref{dinteg}, we deduce
\begin{theorem}
Let $\phi$ be a conformal generic submersion from a K\"{a}hler manifold 
$(M_1,g_1,J)$ to a Riemannian manifold $(M_2,g_2)$ with integrable distribution $({\rm kerd}\phi)^\perp$. If 
$\phi$ is a horizontally homothetic map then we have
\begin{align}
&\lambda^{-2}h\big((\nabla{\rm d}\phi)(\xi,\mathcal{B}\eta)-(\nabla{\rm d}\phi)(\eta,\mathcal{B}\xi)-\nabla^{\phi}_{\xi}{\rm d}\phi(\mathcal{C}\eta)
+\nabla^{\phi}_{\eta}{\rm d}\phi(\mathcal{C}\xi),{\rm d}\phi(\omega W)\big)\\
&=g(-\varphi(\mathcal{V}\nabla^{1}_{\xi}\mathcal{B}\eta-\mathcal{V}\nabla^{1}_{\eta}\mathcal{B}\xi+\mathcal{A}_{\xi}\mathcal{C}\eta-\mathcal{A}_{\eta}\mathcal{C}\xi),W)\notag 
\end{align}
for $\xi,\eta\in\Gamma(({\rm kerd}\phi)^{\perp})$ and $V\in\Gamma({\rm kerd}\phi)$.
\end{theorem}

\begin{remark}
From the above result, one can easily see that a conformal generic submersion with integrable $({\rm kerd}\phi)^\perp$ turns out to
be a horizontally homothetic submersion.
\end{remark}

For the geometry of leaves of the horizontal distribution, we have the
following theorem.

\begin{theorem}\label{ypar}
Let $\phi$ be a conformal generic submersion from a K\"{a}hler
manifold $(M,g,J)$ to a Riemannian manifold $(B,h)$. Then the following conditions are equivalent:
\begin{enumerate}
\item[(i)] the horizontal distribution defines a totally geodesic foliation on $M.$
\item[(ii)] $\lambda^{-2}h((\nabla {\rm d}\phi)(\xi,JV),{\rm d}\phi(\eta))=g(\eta,\mathcal{V}\nabla^{1}_{\xi}JW)$
\item[(iii)] $
\begin{aligned}[t]
&{\lambda^{-2}}h(\nabla^{\phi}_{\xi}{\rm }\phi(\omega V_2),{\rm d}\phi(\mathcal{C}\eta))=-g(\varphi(\mathcal{A}_{\xi}\mathcal{C}\eta+\mathcal{V}\nabla^{1}_{\xi}\mathcal{B}\eta),V_2)\\
&+g(\mathcal{A}_{\xi}\mathcal{B}\eta-\xi(\ln\lambda)\mathcal{C}\eta- \mathcal{C}\eta(\ln\lambda)\xi+g(\xi,\mathcal{C}\eta)\nabla\ln\lambda,\omega V_2) 
\end{aligned}$
\end{enumerate}
for $\xi,\eta\in\Gamma(({\rm kerd}\phi)^{\perp})$ and $V\in\Gamma(\mathcal{D})$.
\end{theorem}
\begin{proof}
Given $\xi,\eta\in\Gamma(({\rm kerd}\phi)^\perp)$ and $JV_1\in\Gamma(\mathcal{D})$, by virtue of (\ref{J})(ii), (\ref{nxv}), (\ref{nxy}), (\ref{jv}) and (\ref{jx}) we obtain
\begin{align*}
g(\nabla^{1}_{\xi}\eta,JV_1)&=-g(\eta, \mathcal{H}\nabla^{1}_{\xi}JV_1+\mathcal{V}\nabla^{1}_{\xi}JV_1)\\
&=\lambda^{-2}h((\nabla {\rm d}\phi)(\xi,JV_1),{\rm d}\phi(\eta))-g(\eta,\mathcal{V}\nabla^{1}_{\xi}JV_1)
\end{align*}
so that 
\begin{align}\label{ypar1}
g(\nabla^{1}_{\xi}\eta,JV_1)=0 \Longleftrightarrow \lambda^{-2}h((\nabla {\rm d}\phi)(\xi,JV_1),{\rm d}\phi(\eta))=g(\eta,\mathcal{V}\nabla^{1}_{\xi}JV_1).
\end{align}
Given $V_2\in\Gamma(\mathcal{D}_2),$ by using (\ref{J})(ii), (\ref{nxv}), (\ref{nxy}), (\ref{jv}) and (\ref{jx}) we get
\begin{align*}
g(\nabla^{1}_{\xi}\eta,V_2)&=-g(\varphi(\mathcal{A}_{\xi}\mathcal{C}\eta+\mathcal{V}%
\nabla^{1}_{\xi}\mathcal{B}\eta),V_2) -g(\mathcal{B}\eta,\nabla^{1}_{\xi}%
\omega V_2)+g(\nabla^{1}_{\xi}\mathcal{C}\eta,\omega V_2)\\
&=-g(\varphi(\mathcal{A}_{\xi}\mathcal{C}\eta+\mathcal{V}%
\nabla^{1}_{\xi}\mathcal{B}\eta),V_2)-g(\mathcal{B}\eta,\mathcal{A}_{\xi}\omega V_2) \\
&-{\lambda^{-2}}g(\nabla \ln\lambda,\xi)h({\rm d}\phi(\omega V_2),{\rm d}\phi(\mathcal{C}\eta)) 
-{\lambda^{-2}}g(\nabla \ln\lambda,\mathcal{C}\eta)h({\rm d}\phi(\xi),{\rm d}\phi(\omega V_2))\\
&+g(\xi,\mathcal{C}\eta){\lambda^{-2}}h({\rm d}\phi(\nabla\ln\lambda),{\rm d}\phi(\omega V_2))
+{\lambda^{-2}}h(\nabla^{\phi}_{\xi}{\rm d}\phi(\mathcal{C}\eta),{\rm d}\phi(\omega V_2))
\end{align*}
so that
\begin{align}  \label{ypar2}
g(\nabla^{1}_{\xi}\eta,V_2)&=g(\mathcal{A}_{\xi}\mathcal{B}\eta-\xi(\ln\lambda)\mathcal{C}\eta- \mathcal{C}\eta(\ln\lambda)\xi+g(\xi,\mathcal{C}\eta)\nabla\ln\lambda,\omega V_2)  \\
&-g(\varphi(\mathcal{A}_{\xi}\mathcal{C}\eta+\mathcal{V}\nabla^{^{M_1}}_{\xi}\mathcal{B}\eta),V_2) -{\lambda^{-2}}h(\nabla^{\phi}_{\xi}{\rm }\phi(\omega V_2),{\rm d}\phi(\mathcal{C}\eta)).\notag
\end{align}
From (\ref{ypar1}) and (\ref{ypar2}) we get $(i) \Longleftrightarrow (ii), (i) \Longleftrightarrow (iii) \ \textrm{and} \ (ii) \Longleftrightarrow (iii)$ which completes the proof.
\end{proof}

From Theorem \ref{ypar}, we immediately deduce

\begin{theorem}
Let $\phi$ be a conformal generic submersion from a K\"{a}hler manifold $%
(M_1,g_{1},J)$ to a Riemannian manifold $(M_2,g_{2})$ with a totally geodesic foliation $({\rm kerd}\phi)^\perp.$ If $\phi$ is a horizontally homothetic map.
then we have
\begin{align}\label{ypar3}
\lambda^{-2}h(\nabla^{\phi}_{\xi}{\rm d}\phi(\omega V_2),{\rm d}\phi(\mathcal{C}\eta))&=
g(\mathcal{A}_{\xi}\mathcal{B}\eta,\omega V_2)-g(\varphi(\mathcal{A}_{\xi}\mathcal{C}\eta+\mathcal{V}\nabla^{1}_{\xi}\mathcal{B}\eta),V_2)\}
\end{align}
for any $\xi,\eta\in\Gamma(({\rm kerd}\phi)^\perp)$ and $V\in\Gamma({\rm kerd}\phi)$.
\end{theorem}
\begin{proof}
Since $(ker\phi_{*})^\perp$ defines a totally geodesic foliation on $M_1$, from  (\ref{ypar2}) we have
\begin{align*}  
g(\nabla^{1}_{\xi}\eta,V_2)&=g(\mathcal{A}_{\xi}\mathcal{B}\eta-\xi(\ln\lambda)\mathcal{C}\eta- \mathcal{C}\eta(\ln\lambda)\xi+g(\xi,\mathcal{C}\eta)\nabla\ln\lambda,\omega V_2)  \\
&-g(\varphi(\mathcal{A}_{\xi}\mathcal{C}\eta+\mathcal{V}\nabla^{^{M_1}}_{\xi}\mathcal{B}\eta),V_2) -{\lambda^{-2}}h(\nabla^{\phi}_{\xi}{\rm }\phi(\omega V_2),{\rm d}\phi(\mathcal{C}\eta))\notag
\end{align*}
for any $\xi,\eta\in\Gamma(({\rm kerd}\phi)^\perp)$ and $V_2\in\Gamma({\rm kerd}\phi).$ 
Now, one can easily see that if $\lambda$ is a constant on $({\rm kerd}\phi)^\perp$, we obtain (\ref{ypar3}.)
\end{proof}

In the sequel we are going to investigate the geometry of leaves of the
distribution ${\rm kerd}\phi$.

\begin{theorem} \label{dpar}
 Let $\phi$ be a conformal generic submersion from a K\"{a}hler
manifold $(M,g,J)$ to a Riemannian manifold $(B,h)$. Then the vertical distribution defines a totally geodesic foliation on $M$
if and only if
$$\mathcal{T}_U{\varphi V}+\mathcal{H}\nabla^{1}_U{\omega V}\in\Gamma(\mu), \hat{\nabla}_U{\varphi V}+\mathcal{T}_U{\omega V}\in\Gamma(D_1)$$
\textrm{and} 
\begin{align*}
\lambda^{-2}h(\nabla^{\phi}_{\omega V}{\rm d}\phi(\omega U),{\rm d}\phi \xi)&=
g(-\mathcal{C}\mathcal{T}_{U}\varphi V-\mathcal{A}_{\omega V}\varphi U-g(\omega V,\omega U)\nabla(\ln\lambda),\xi)
\end{align*}
for any $U,V\in\Gamma({\rm kerd}\phi)$ and $\xi\in\Gamma(\mu).$
\end{theorem}
\begin{proof}
Given any $U,V\in\Gamma({\rm kerd}\phi)$ and $\xi\in\Gamma(\mu)$, by
using (\ref{J})(ii) and (\ref{jv}) we get
\begin{align*}
g(\nabla^{1}_UV,\xi)=g(\nabla^{1}_U{\varphi}V+\nabla^{1}_{U}{\omega}V,J\xi).
\end{align*}
Now, by using (\ref{nvw}) we have
\begin{align*}
g(\nabla^{1}_UV,\xi)=g(-\mathcal{C}\mathcal{T}_{U}\varphi V,\xi)-g(\mathcal{A}_{\omega V}\varphi U,\xi)-\lambda^{-2}h({\rm d}\phi(\nabla^{1}_{\omega V}\omega U),{\rm d}\phi \xi)
\end{align*}
so that
\begin{align}\label{dpar1}
g(\nabla^{1}_UV,\xi)&=g(-\mathcal{C}\mathcal{T}_{U}\varphi V,\xi)-g(\mathcal{A}_{\omega V}\varphi U,\xi)-\lambda^{-2}h(\nabla^{{\rm d}\phi}_{\omega V}{\rm d}\phi(\omega U)
,{\rm d}\phi \xi)\\
&-g(\omega V,\omega U)\lambda^{-2}h({\rm d}\phi(\nabla\ln\lambda),{\rm d}\phi(\xi))\notag
\end{align}
which tells that
\begin{align*}
g(\nabla^{1}_UV,\xi)&=g(-\mathcal{C}\mathcal{T}_{U}\varphi V-\mathcal{A}_{\omega V}\varphi U-g(\omega V,\omega U)\nabla(\ln\lambda),\xi)\\
&-\lambda^{-2}h(\nabla^{\phi}_{\omega V}{\rm d}\phi(\omega U),{\rm d}\phi \xi).
\end{align*}
Given for any $U,V\in\Gamma({\rm kerd}\phi)$ and $Z\in\Gamma(\mathcal{D}^\perp)$, by using (\ref{J})(ii), (\ref{jv}) and (\ref{nvw}) we get
\begin{align*}
g(\nabla^{1}_UV,\omega Z)&=-g(J\nabla^{1}_U{JV},\omega Z)\\
&=-g(J(\mathcal{T}_U{\varphi V}+\hat{\nabla}_U{\varphi V}+\mathcal{T}_U{\omega V}+\mathcal{H}\nabla^{1}_U{\omega V}),\omega Z)\\
&=-g(\mathcal{C}(\mathcal{T}_U{\varphi V}+\mathcal{H}\nabla^{1}_U{\omega V})+\omega(\hat{\nabla}_U{\varphi V}+\mathcal{T}_U{\omega V}),\omega Z)
\end{align*}
\end{proof}
From Theorem \ref{dpar}, we have
\begin{theorem}
Let $\phi$ be a conformal generic submersion from a K\"{a}hler manifold $%
(M,g,J)$ to a Riemannian manifold $(B,h)$. Then any two
conditions below imply the third;
\begin{enumerate}
\item[(i)] The vertical distribution defines a totally geodesic foliation on $M$.
\item[(ii)] $\lambda$ is a constant on $\Gamma(\mu)$.
\item[(iii)] $%
\begin{aligned}[t]
\lambda^{-2}h(\nabla^{\phi}_{\omega V}{\rm d}\phi(\omega U),{\rm d}\phi \xi)&=
g(-\mathcal{C}\mathcal{T}_{U}\varphi V\xi)+g(\mathcal{A}_{\omega V}\xi,\varphi U)
\end{aligned}$
\end{enumerate}
for $U,V\in\Gamma(ker\phi_{*})$ and $X\in\Gamma((ker\phi_{*})^\perp)$.
\end{theorem}
\begin{proof}
In view of Eq. (\ref{dpar1}), if we have (i) and (iii), then we have that $$g(\omega V,\omega U)g(\nabla(\ln\lambda),\xi)=0,$$
which tells that $\lambda$ is a constant on $\Gamma(\mu)$. One can easily get the other assertions.
\end{proof}

From Theorem \ref{ypar} and Theorem \ref{dpar}, we have the following result.
\begin{theorem}
Let $\phi:(M,g,J)\longrightarrow (B,h)$ be a conformal generic
submersion from a K\"{a}hler manifold $(M,g,J)$ onto a Riemannian
manifold $(B,h)$. Then the total space $M$ is a generic product
manifold of the leaves of ${\rm kerd}\phi$ and $({\rm kerd}\phi)^\perp,$ i.e., $M=M{%
_{{\rm kerd}\phi}}\times M{_{({\rm kerd}\phi)^\perp}},$ if and only if
$$\lambda^{-2}h((\nabla {\rm d}\phi)(\xi,JV),{\rm d}\phi(\eta))=g(\eta,\mathcal{V}\nabla^{1}_{\xi}JW),$$
\begin{align*}
&{\lambda^{-2}}h(\nabla^{\phi}_{\xi}{\rm }\phi(\omega V_2),{\rm d}\phi(\mathcal{C}\eta))=-g(\varphi(\mathcal{A}_{\xi}\mathcal{C}\eta+\mathcal{V}\nabla^{1}_{\xi}\mathcal{B}\eta),V_2)\\
&+g(\mathcal{A}_{\xi}\mathcal{B}\eta-\xi(\ln\lambda)\mathcal{C}\eta- \mathcal{C}\eta(\ln\lambda)\xi+g(\xi,\mathcal{C}\eta)\nabla\ln\lambda,\omega V_2) 
\end{align*}
\textrm{and}
$$\mathcal{T}_U{\varphi V}+\mathcal{H}\nabla^{1}_U{\omega V}\in\Gamma(\mu), \hat{\nabla}_U{\varphi V}+\mathcal{T}_U{\omega V}\in\Gamma(D_1),$$
\begin{align*}
\lambda^{-2}h(\nabla^{\phi}_{\omega V}{\rm d}\phi(\omega U),{\rm d}\phi \xi)&=
g(-\mathcal{C}\mathcal{T}_{U}\varphi V-\mathcal{A}_{\omega V}\varphi U-g(\omega V,\omega U)\nabla(\ln\lambda),\xi)
\end{align*}
for any $\xi,\eta\in\Gamma(({\rm kerd}\phi)^\perp),$ $U,V,V_2\in\Gamma({\rm kerd}\phi)$,
where $M{_{{\rm kerd}\phi}}$ and $M{_{({\rm kerd}\phi)^\perp}}$ are leaves of the
distributions ${\rm kerd}\phi$ and $({\rm kerd}\phi)^\perp$, respectively.
\end{theorem}

\section{\textbf{Totally geodesicity and Harmonicity of conformal generic submersions}}

In this section, we investigate the necessary and sufficient conditions for
such submersions to be totally geodesicity and harmonicity, respectively. We first give the following  definition.

\subsection{Totally geodesicity of $\phi:(M,g,J)\longrightarrow (B,h)$} 
\begin{definition}
Let $\phi$ be a conformal generic submersion from a K\"{a}hler manifold $%
(M,g,J)$ to a Riemannian manifold $(B,h)$. Then $\phi$ is called
a $(\omega \mathcal{D}^\perp,\mu)$-totally geodesic map if
\begin{equation*}
(\nabla {\rm d}\phi)(\omega Z,\xi)=0, for \ Z\in\Gamma(\mathcal{D}^\perp)\ and \ \xi\in\Gamma(\mu).
\end{equation*}
\end{definition}

The following result show that the above definition has an important effect on the character of the conformal generic submersion.

\begin{theorem}
Let $\phi$ be a conformal generic submersion from a K\"{a}hler manifold $%
(M,g,J)$ to a Riemannian manifold $(B,h)$. Then $\phi$ is a $(\omega \mathcal{D}^\perp,\mu)$-totally geodesic map if and only if $\phi$ is a
\textit{horizontally homotetic map}. Then the following conditions are equivalent:
\begin{enumerate}
\item[(i)] $\phi$ is a \textit{horizontally homothetic map}.
\item[(ii)] $\phi$ is a $(\omega \mathcal{D}^\perp,\mu)$-totally geodesic map. 
\end{enumerate}
\end{theorem}
\begin{proof}
Given $Z\in\Gamma(\mathcal{D}^\perp)$ and $\xi\in\Gamma(\mu)$, by Lemma \ref{lem1}, we
have
\begin{align*}
(\nabla {\rm d}\phi)(\omega Z,\xi)&=\omega Z(\ln\lambda){\rm d}\phi(\xi)+\xi(\ln\lambda){\rm d}\phi(\omega Z)-g(\omega Z,\xi){\rm d}\phi(\nabla\ln\lambda).\\
&=\omega Z(\ln\lambda){\rm d}\phi(\xi)+\xi(\ln\lambda){\rm d}\phi(\omega Z).
\end{align*}
From above equation, we easily get $(i) \Longrightarrow (ii).$ Conversely, if $(\nabla {\rm d}\phi)(\omega Z,\xi)=0,$ we get
\begin{equation}
\omega Z(\ln\lambda){\rm d}\phi(\xi)+\xi(\ln\lambda){\rm d}\phi(\omega Z)=0.  \label{e.q:4.2}
\end{equation}
From above equation, since \{${\rm d}\phi(\xi), {\rm d}\phi(\omega Z)$\}  is linearly independent for non-zero $\xi, Z=\{0\},$ we have $\omega Z(\ln\lambda)=0$ and $\xi(\ln\lambda).$ It means that $\lambda$ is a constant on $\Gamma(\mathcal{D}^\perp)$ and $\Gamma(\mu),$ which gives that $(i) \Longleftarrow (ii).$ This completes the proof of the theorem.
\end{proof}

We also have the following result.

\begin{theorem}
Let $\phi:(M,g,J)\longrightarrow (B,h)$ is a conformal generic
submersion, where $(M,g,J)$ is a K\"{a}hler manifold and $(B,h)$ is
a Riemannian manifold. Then the following conditions are equivalent:
\begin{enumerate}
\item[(i)] $\phi$ is a totally geodesic map.
\item[(ii)] $\mathcal{C}\mathcal{T}_{U}JV+\omega\hat{\nabla}_{U}JV=0 for \;U,V\in\Gamma(\mathcal{D}).$
\item[(iii)] $\mathcal{T}_{U}\varphi Z+\mathcal{A}_{\omega Z}U\in\Gamma(\omega \mathcal{D}^\perp)$ \textrm{and}
$\hat{\nabla}_{U}\varphi Z+\mathcal{T}_{U}\omega Z\in\Gamma(\mathcal{D}),$ for $U\in\Gamma(\mathcal{D}),\;Z\in\Gamma(\mathcal{D}^\perp).$
\item[(iv)] $\mathcal{T}_{V}\mathcal{B}\xi+\mathcal{H}\nabla^{1}_{V}\mathcal{C}\xi\in\Gamma(\omega\mathcal{D}^\perp)$
\;\textrm{and}\; $\hat{\nabla}_{V}\mathcal{B}\xi+\mathcal{T}_{V}\mathcal{C}\xi\in\Gamma(J\mathcal{D})$, for $V\in\Gamma({\rm kerd}\phi),\; \xi\in\Gamma(({\rm kerd}\phi)^\perp)$
\item[(v)] $\phi$ is a horizontally homotetic map.
\end{enumerate}
\end{theorem}

\begin{proof}
 In view of Eq. (\ref{J})(ii) and Eq. (\ref{nfixy}) we have
\begin{align*}
(\nabla {\rm d}\phi)(U,V)={\rm d}\phi(J\nabla^{1}_{U}JV)
\end{align*}
for any $U,V\in\Gamma(\mathcal{D}).$ Then from Eq. (\ref{nvw}) we arrive at
\begin{equation*}
(\nabla {\rm d}\phi)(U,V)={\rm d}\phi(J(\mathcal{T}_{U}JV+\hat{\nabla}_{U}JV)).
\end{equation*}
Using Eq. (\ref{jv}) and Eq. (\ref{jx}) in above equation we obtain
\begin{equation*}
(\nabla {\rm d}\phi)(U,V)={\rm d}\phi(\mathcal{B}\mathcal{T}_{U}JV+\mathcal{C}\mathcal{T}_{U}JV+\varphi\hat{\nabla}_{U}JV+\omega\hat{\nabla}_{U}JV).
\end{equation*}
So 
\begin{equation}
(\nabla {\rm d}\phi)(U,V)=0 \Longleftrightarrow \mathcal{C}\mathcal{T}_{U}JV+\omega\hat{\nabla}_{U}JV=0.\label{tgm1}
\end{equation}

Given $U\in\Gamma({\rm kerd}\phi),\ Z\in\Gamma(\mathcal{D}^\perp)$, by  Eq. (\ref{J})(ii) and Eq. (\ref{nfixy}) we have
\begin{align*}
(\nabla {\rm d}\phi)(U,Z)={\rm d}\phi(J\nabla^{1}_{U}JZ).
\end{align*}
By Eq. (\ref{nvw}), Eq. (\ref{nvx}) and Eq. (\ref{jv}) yields
\begin{equation*}
(\nabla {\rm d}\phi)(U,Z)={\rm d}\phi(J(\mathcal{T}_{U}\varphi Z+\hat{\nabla}_{U}\varphi Z+\mathcal{T}_{U}\omega Z+\mathcal{A}_{\omega Z}U)).
\end{equation*}
where we have used $\mathcal{H}\nabla_{\omega Z}U=\mathcal{A}_{\omega Z}U.$ By using Eq. (\ref{jv}) and Eq. (\ref{jx}) in above equation we obtain
\begin{align*}
(\nabla {\rm d}\phi)(U,Z)&={\rm d}\phi(\mathcal{B}\mathcal{T}_{U}\varphi Z+\mathcal{C}\mathcal{T}_{U}\varphi Z
+\varphi\hat{\nabla}_{U}\varphi Z+\omega\hat{\nabla}_{U}\varphi Z \\
&+\varphi \mathcal{T}_{U}\omega Z+\omega \mathcal{T}_{U}\omega Z+\mathcal{B}\mathcal{A}_{\omega Z}U+\mathcal{C}\mathcal{A}_{\omega Z}U).
\end{align*}
So 
\begin{equation}
(\nabla {\rm d}\phi)(U,Z)=0\Longleftrightarrow \mathcal{C}(\mathcal{T}_{U}\varphi Z+\mathcal{A}_{\omega Z}U)+\omega(\hat{\nabla}_{U}\varphi Z+\mathcal{T}_{U}\omega Z)=0.\label{tgm2}
\end{equation}

Given $V\in\Gamma({\rm kerd}\phi),\ \xi\in\Gamma(({\rm kerd}\phi)^\perp)$, by  Eq. (\ref{J})(ii), Eq. (\ref{nfixy}), Eq. (\ref{nvw}), Eq. (\ref{nvx}) and Eq. (\ref{jv}) yields
\begin{align*}
(\nabla {\rm d}\phi)(V,\xi)&={\rm d}\phi(J\nabla^{1}_{V}J\xi).\\
&={\rm d}\phi(J(\nabla^{1}_{V}\mathcal{B}\xi+\nabla^{1}_{V}\mathcal{C}\xi)).\\
&={\rm d}\phi(J(\mathcal{T}_{V}\mathcal{B}\xi+\hat{\nabla}_{V}\mathcal{B}\xi+\mathcal{T}_{V}\mathcal{C}\xi+\mathcal{H}\nabla^{1}_{V}\mathcal{C}\xi)).\\
&={\rm d}\phi(\mathcal{C}(\mathcal{T}_{V}\mathcal{B}\xi+\mathcal{H}\nabla^{1}_{V}\mathcal{C}\xi)+\omega(\hat{\nabla}_{V}\mathcal{B}\xi+\mathcal{T}_{V}\mathcal{C}\xi)).
\end{align*}
So  
\begin{equation}
(\nabla {\rm d}\phi)(V,\xi)=0 \Longleftrightarrow \mathcal{T}_{V}\mathcal{B}\xi+\mathcal{H}\nabla^{1}_{V}\mathcal{C}\xi\in\Gamma(\omega\mathcal{D}^\perp)
\;\textrm{and}\; \hat{\nabla}_{V}\mathcal{B}\xi+\mathcal{T}_{V}\mathcal{C}\xi\in\Gamma(J\mathcal{D}).\label{tgm3}
\end{equation}

Now, we will show that for any $\xi,\eta\in\Gamma(\mu)$, $(\nabla {\rm d}\phi)(\xi,\eta)=0 \Longleftrightarrow \phi$ is a horizontally homothetic map.

Given $\xi,\eta\in\Gamma(\mu)$, from Lemma \ref{lem1}, we have
\begin{equation*}
(\nabla {\rm d}\phi)(\xi,\eta)=\xi(\ln\lambda){\rm d}\phi(\eta)+\eta(\ln\lambda){\rm d}\phi(\xi)-g(\xi,\eta){\rm d}\phi(\nabla\ln\lambda).
\end{equation*}
Taking $\eta=J\xi,\ \xi\in\Gamma(\mu)$ in the above equation we get
\begin{align*}
(\nabla {\rm d}\phi)(\xi,J\xi)&=\xi(\ln\lambda){\rm d}\phi(J\xi)+J\xi(\ln\lambda){\rm d}\phi(\xi)-g(\xi,J\xi){\rm d}\phi(\nabla\ln\lambda) \\
&=\xi(\ln\lambda){\rm d}\phi(J\xi)+J\xi(\ln\lambda){\rm d}\phi\xi.
\end{align*}
If $(\nabla {\rm d}\phi)(\xi,J\xi)=0,$ we get
\begin{equation}
\xi(\ln\lambda){\rm d}\phi(J\xi)+J\xi(\ln\lambda){\rm d}\phi\xi=0.  \label{tgm4}
\end{equation}
Taking inner product in Eq. (\ref{tgm4}) with ${\rm d}\phi(\xi)$ and taking into account $\phi$ is a conformal submersion, we have
\begin{equation*}
g(\nabla\ln\lambda,\xi)h({\rm d}\phi J\xi,{\rm d}\phi\xi)+g(\nabla\ln\lambda,J\xi)h({\rm d}\phi\xi,{\rm d}\phi\xi)=0.
\end{equation*}
which implies that $\lambda$ is a constant on $\Gamma(J\mu).$ On the other hand, taking inner product in Eq. (\ref{tgm4}) with ${\rm d}\phi (J\xi)$ we have
\begin{equation*}
g(\nabla\ln\lambda,\xi)h({\rm d}\phi J\xi,{\rm d}\phi J\xi)+g_{1}(\nabla\ln\lambda,\xi)h({\rm d}\phi \xi,{\rm d}\phi J\xi)=0.
\end{equation*}
which tells that $\lambda$ is a constant $\Gamma(\mu).$
In a similar way, for $U,V\in\Gamma(\mathcal{D}^\perp)$, by using Lemma \ref{lem1} we have
\begin{equation*}
(\nabla {\rm d}\phi)(\omega U,\omega V)=\omega U(\ln\lambda){\rm d}\phi(\omega V)+\omega V(\ln\lambda){\rm d}\phi(\omega U)
-g(\omega U,\omega V){\rm d}\phi(\nabla\ln\lambda).
\end{equation*}
From above equation, taking $V=U$ we obtain
\begin{align}
(\nabla {\rm d}\phi)(\omega U,\omega U)&=2\omega U(\ln\lambda){\rm d}\phi(\omega U)-g(\omega U,\omega U){\rm d}\phi(\nabla\ln\lambda). \label{tgm5} 
\end{align}
Taking inner product in Eq. (\ref{tgm5}) with ${\rm d}\phi(\omega U)$ and taking into account $\phi $ is a conformal submersion, we derive
\begin{equation*}
2g(\nabla\ln\lambda,\omega U)h({\rm d}\phi(\omega U),{\rm d}\phi(\omega U))
-g(\omega U,\omega U)h({\rm d}\phi(\nabla\ln\lambda),{\rm d}\phi(\omega U))=0
\end{equation*}
which tells that $\lambda$ is a constant on $\Gamma(\omega \mathcal{D}^\perp)$. Thus $\lambda$ is a constant on $\Gamma(({\rm kerd}\phi)^\perp)$.
By Eq. (\ref{tgm1}), Eq. (\ref{tgm2}), Eq. (\ref{tgm3}), Eq. (\ref{tgm4}), we have $(i) \Longleftrightarrow (ii), (i) \Longleftrightarrow (iii), (i)\Longleftrightarrow(iv), (i)\Longleftrightarrow(v).$
This completes the proof of the theorem.
\end{proof}

\subsection{Harmonicity of $\phi:(M,g,J)\longrightarrow (B,h)$}

Let  $\phi:N_1\longrightarrow N_2$ be a $C^{\infty}$ map between two Riemannian manifolds. We can naturally define a function $e(\phi)=N_1\longrightarrow [0, \infty]$ given by
$$e(\phi)(x)=\frac{1}{2}|(d\phi)_x|, x\in N_2$$
where $|(d\phi)_x|$ denotes the Hilbert-Schmidt norm of $(d\phi)_x$. We call $e(\phi)$ the energy density of $\phi$. Let $\Omega$ is the compact closure $\bar{U}$ of a non empty connected open subset $U$ of $N_1$. The energy integral of $\phi$ over $\Omega$ is the integral of its energy density:
$$E(\phi;\Omega) =\int_{\Omega}e(\phi) v_{g_N}=\int_{\Omega}\frac{1}{2}|(d\phi)_x| v_{g_N}$$
where $v_{g_N}$ is the volume form on $(N,g_N)$. Let $C^{\infty}(N_1,N_2)$ denote the space of all
differentiable map from $N_1$ on $N_2$. A differentiable map $\phi:N_1\longrightarrow N_2$ is said to harmonic if it is a critical point of the energy functional 
$E(\phi;\Omega):C^{\infty}(N_1,N_2)\longrightarrow \mathbb{R}$ for any compact domain $\Omega\subset N_1.$ By the result of J. Eells and J. Sampson [6], we know that the map $\phi$ is harmonic if and only if the tension field $$\tau(\phi) =trace(\nabla d\phi)=0.$$

\begin{theorem}
Let $\phi:(M,g,J)\longrightarrow (B,h)$ is a conformal generic
submersion, where $(M,g,J)$ is a K\"{a}hler manifold and $(B,h)$ is
a Riemannian manifold. Then $\phi$ is harmonic if and only if 
\begin{eqnarray*}
&&trace |_{(D_1)}-d\phi\Big(\mathcal{C}\mathcal{T}_{J(.)}(.)+\omega\hat{\nabla}_{J(.)}(.)\Big)+\\
&&trace |_{(D_2)} \ d\phi\Big(\mathcal{C}\mathcal{T}_{(.)}\varphi (.)+\omega\hat{\nabla}_{(.)}\varphi (.)+\omega \mathcal{T}_{(.)}\omega (.)
+\mathcal{C}\mathcal{H}\nabla^{1}_{(.)}\omega (.)\Big)\\
&&=trace |_{(kerd\phi)^{\perp}} \nabla^{\phi}_{(.)}d\phi(\mathcal{C}^{2}(.)+\omega\mathcal{B}(.))-\\
&&d\phi\Big(\mathcal{C}\mathcal{A}_{(.)}\mathcal{B}(.)+\mathcal{C}\mathcal{H}\nabla^{1}_{(.)}\mathcal{C}(.)+\omega \mathcal{A}_{(.)}\mathcal{C}(.)+\omega\mathcal{V}\nabla^{1}_{(.)}\mathcal{C}(.)\Big)
\end{eqnarray*}
\end{theorem}
\begin{proof}
For any $U\in\Gamma(\mathcal{D}_1),$ $V\in\Gamma(\mathcal{D}_2)$ and $\xi\in\Gamma(({\rm kerd}\phi)^\perp),$ by using Eq. (\ref{J})(i),  Eq. (\ref{nfixy}), Eq. (\ref{jx}), Eq. (\ref{jv}), and Proposition 3.1 (f) we have
\begin{eqnarray*}
&&(\nabla d\phi)(JU,JU)+(\nabla d\phi)(V,V)+(\nabla d\phi)(\xi,\xi)=-d\phi(J\nabla^{1}_{JU}V)\\
&&+d\phi(J(\nabla^{1}_{V}\varphi V+\nabla^{1}_{V}\omega V))-\nabla^{\phi}_{\xi}d\phi(\mathcal{C}^{2}\xi+\omega\mathcal{B}\xi)+d\phi(J(\nabla^{1}_{\xi}\mathcal{B}\xi
+\nabla^{1}_{\xi}\mathcal{C}\xi)).
\end{eqnarray*}
A straight computation by using Eq. (\ref{jx}), Eq. (\ref{jv}) and Eq. (\ref{nvw})-(\ref{nxy}), we obtain
\begin{eqnarray*}
&&(\nabla d\phi)(JU,JU)+(\nabla d\phi)(V,V)+(\nabla d\phi)(\xi,\xi)=-d\phi(\mathcal{C}\mathcal{T}_{JU}U+\omega\hat{\nabla}_{JU}U)\\
&&+d\phi(\mathcal{C}\mathcal{T}_{V}\varphi V+\omega\hat{\nabla}_{V}\varphi V+\omega \mathcal{T}_{V}\omega V+\mathcal{C}\mathcal{H}\nabla^{1}_{V}\omega V)\\
&&-\nabla^{\phi}_{\xi}d\phi(\mathcal{C}^{2}\xi+\omega\mathcal{B}\xi)+d\phi(\mathcal{C}\mathcal{A}_{\xi}\mathcal{B}\xi+\mathcal{C}\mathcal{H}\nabla^{1}_{\xi}\mathcal{C}\xi 
+\omega \mathcal{A}_{\xi}\mathcal{C}\xi+\omega\mathcal{V}\nabla^{1}_{\xi}\mathcal{C}\xi).
\end{eqnarray*}
Now,  by taking trace on the above equation, we obtain the proof of the theorem.
\end{proof}

\begin{remark}
One can easily see that the maps defined Example \ref{exm1} and Example \ref{exm2} are an example of harmonic map.
\end{remark}

\end{document}